\documentclass[a4paper,12pt]{amsart}

\usepackage{graphicx}
\usepackage{amsmath}
\usepackage{amssymb}
\usepackage{oldgerm}
\usepackage{multicol}
\usepackage[all]{xy}
\usepackage{amsthm}
\usepackage{setspace}

\usepackage{enumerate}

\newtheorem{thm}{Theorem}[section]
\newtheorem{cor}[thm]{Corollary}

\newtheorem{prop}[thm]{Proposition}

\theoremstyle{definition}
\newtheorem{defin}[thm]{Definition}
\newtheorem{rem}[thm]{Remark}
\newtheorem{exa}[thm]{Example}

\numberwithin{equation}{section}

\begin{document}
\baselineskip=15pt


\title[Existence of Gradings on Associative Algebras]{Existence of Gradings on Associative Algebras}

\author{Dusko Bogdanic}

\email{dusko.bogdanic@gmail.com}
\date{}

\begin{abstract} 

In this paper we study the existence of gradings on finite dimensional associative algebras. We prove that a connected algebra $A$ does not have a non-trivial grading if and only if  $A$ is basic,  its quiver has one vertex, and its group of outer automorphisms  is unipotent. We apply this result to prove that up to graded Morita equivalence there do not exist non-trivial gradings  on the blocks of group algebras with quaternion defect groups and one isomorphism class of simple modules.
\end{abstract}

\subjclass[2010]{Primary 16A03; Secondary 16W50, 20C20}

\keywords{Graded algebras, graded Morita equivalence, quaternion defect groups, unipotent morphisms.}

\maketitle

\section{Introduction and preliminaries}

Let $A$ be an algebra over a field $k$. We say that  $A$ is a graded algebra if $A$ is  the direct sum of subspaces $A=\bigoplus_{i\in\mathbb{Z}}
A_i$, such that $A_iA_j\subset A_{i+j}$, $i,j\in \mathbb{Z}$. The subspace $A_i$ is said to be the homogeneous subspace of degree $i$. It is obvious that we can always trivially grade $A$ by setting $A_0=A$. In this paper we study the problem of existence of non-trivial gradings on a given algebra. For some algebras it is not too difficult to construct a non-trivial grading, e.g.\  a polynomial algebra $k[x]$ can be graded by declaring $1$ to be a homogeneous element of degree $0$, and $x$ to be a homogeneous element of degree $d$, where $d\in \mathbb{Z}$ is arbitrary. For some other algebras, such as group algebras, it is not obvious how one can construct non-trivial gradings on these algebras. More complex methods, such as transfer of gradings via derived and stable equivalences (\cite{Rou, GBTA, BDeh}), had to be developed to introduce non-trivial gradings on certain blocks of group algebras. In this paper we focus our attention on the following question. How do we prove that a given algebra does not possess a non-trivial grading? If there is not an obvious generating set of $A$ that would consist of homogeneous elements, as in the case of $k[x]$ (where $\{1,x\}$ was a generating set of homogeneous elements), how do we prove that none of the generating sets (or bases) could serve as a homogeneous generating set (or basis)?  In the next section of this paper we prove that a connected finite dimensional algebra $A$ can not be non-trivially graded if and only if it is basic,  its quiver has one vertex, and its group of outer automorphisms ${\rm Out}(A)$ is unipotent. In the last section, we apply this result on the blocks of group algebras with quaternion defect groups and one isomorphism class of simple modules, and prove that they can only be trivially graded, up to graded Morita equivalence.

If $A$ is a graded algebra, then an $A$-module $M$ is graded if it is the direct sum of subspaces
$M=\bigoplus_{i\in\mathbb{Z}} M_i$,  such that  $A_iM_j\subset
M_{i+j}$, for all $i,j\in \mathbb{Z}$. If $M$ is a graded
$A$--module, then $N=M\langle i\rangle$ denotes the shifted graded module
given by $N_j=M_{i+j}$, $j\in \mathbb{Z}$. An $A$-module
homomorphism $f$ between two graded modules $M$ and $N$ is a
homomorphism of graded modules if $f(M_i)\subseteq N_i$, for all
$i\in \mathbb{Z}$.  We set ${\rm Homgr}_A(M,N):=\bigoplus_{i\in
\mathbb{Z}}{\rm Hom}_{A-{\rm gr}}(M,N\langle i\rangle),$ where ${\rm
Hom}_{A-{\rm gr}}(M,N\langle i \rangle)$ denotes the space of all graded
homomorphisms between $M$ and $N\langle i\rangle$ (the space of
homogeneous morphisms of degree $i$). For finitely generated $A$-modules $M$ and $N$ there is an isomorphism of
vector spaces ${\rm Hom}_A(M,N)\cong {\rm Homgr}_A(M,N)$ that
gives us a grading on ${\rm Hom}_A(M,N)$  (see \cite[Corollary 2.4.4]{NFV}). 

\begin{exa}\label{ex:quiverGrading}
{\rm Let $A$ be a finite dimensional algebra given by the quiver
$Q$ and the ideal of relations $I$, i.e.\ $A= kQ/I$, where $I$ is
an admissible ideal of $kQ$ (we recommend \cite{AS} as a good introduction to
path algebras of quivers). Since the path algebra $kQ$ is generated by the vertices and arrows of $Q$, in order to grade $kQ$ it is sufficient to declare vertices and arrows to be homogeneous elements, and to define the degrees of the arrows since the 
vertices of $Q$ will be in degree 0. In order to grade $kQ/I$ it
is sufficient to ensure that $I$ is a homogeneous ideal of $kQ$.

}
\end{exa}

If $A$ is a graded algebra  and $e$ is a homogeneous primitive idempotent, then $Ae$ is going to be a graded projective indecomposable module. Moreover, ${\rm top}\,Ae\,$ will be in degree 0. We note here that if we have two different gradings on an indecomposable module, then they differ only by
a shift (see \cite[Lemma 2.5.3]{BGS}).  Shifting gradings on projective indecomposable modules leads us to the notion of graded Morita  equivalence.


\section{The existence of non-trivial gradings}

From now on we assume that $k$ is an algebraically closed field. For a finite dimensional algebra $A$ over  $k$,  there  is a correspondence between gradings on $A$ and homomorphisms of
algebraic groups from $\textbf{G}_m$ to ${\rm Aut}(A)$, where
$\textbf{G}_m$ is the multiplicative group $k^*$ of the field $k$, and ${\rm Aut}(A)$ is the group of  automorphisms of $A$.
For each grading $A=\bigoplus_{i\in \mathbb{Z}}A_i$ there is a
homomorphism of algebraic groups $\pi \, : \, \textbf{G}_m
\rightarrow {\rm Aut}(A)$ where an element $x\in k^*$ acts on
$A_i$ by multiplication by $x^i$, and vice versa, for a
homomorphism of algebraic groups $\pi \, : \, \textbf{G}_m
\rightarrow {\rm Aut}(A)$, there is a grading $A=\bigoplus_{i\in \mathbb{Z}}A_i$, where $A_i=\{a\in A\, | \,\pi(x)(a)=x^ia, \, {\rm for\,\, all}\,\,  x\in \textbf{G}_m  \}$. If
$A$ is graded and $\pi$ is the corresponding homomorphism, we will
write $(A,\pi)$ to denote that $A$ is graded with grading $\pi$.

\begin{defin}[{\cite[Section 5]{Rou}}]\label{grMorDef} Let $(A,\pi)$ and $(A,\pi^{\prime})$ be gradings
on a finite dimensional $k$-algebra $A$,  and let
$S_1,S_2,\dots,S_r$ be the isomorphism classes of simple
$A$-modules. We say that $(A,\pi)$ and $(A,\pi^{\prime})$ are
graded Morita equivalent if there exist integers $d_{ij}$, where
$1\leq j\leq {\rm dim }\, S_i$ and $1\leq i\leq r$, such that the
graded algebras $(A,\pi^{\prime})$ and ${\rm
Endgr}_{(A,\pi)}(\bigoplus_{i,j}P_i\langle d_{ij}\rangle)^{\rm
op}$ are isomorphic, where $P_i$ denotes the projective cover of
$S_i$.
\end{defin}
Note that graded algebras $(A,\pi)$ and $(A,\pi^{\prime})$ are
graded Morita equivalent if and only if their categories of graded
modules are equivalent. 

In the next two propositions we give necessary conditions for a connected finite-dimensional algebra $A$ not to possess a non-trivial grading. 

\begin{prop}\label{basic} If $A$ is a finite dimensional $k$-algebra which does not have non-trivial gradings, then $A$ is a basic algebra.
\end{prop}
\begin{proof} Let us assume that $A$ is not a basic algebra. There is at least one simple $A$-module, say $S_0$, whose dimension is greater than 1. If $P_0$ is the projective cover of $S_0$, then $P_0$ appears at least twice as a summand of $A$, i.e.\ we have a decomposition $A=P_0\oplus P_0\oplus P$, where $P$ is a projective module. If we look at $A$ as a trivially graded algebra, then the graded algebra $A^{\prime}={\rm Endgr}_{A}(P_0\langle d_1 \rangle \oplus P_0\langle d_2 \rangle \oplus P\langle d_3 \rangle)^{op} $ is graded Morita equivalent to the trivially graded algebra $A$. If we choose $d_1$ and $d_2$ to be different integers, then $A^{\prime}$ is a non-trivially graded algebra, because the identity map from $P_0\langle d_1 \rangle$  to  $P_0\langle d_2 \rangle$ is a homogeneous element of degree $d_1-d_2\neq 0$. Thus, we obtained a non-trivial grading on $A$. 
\end{proof}

The non-existence of non-trivial gradings on a basic algebra imposes a strong condition on its quiver.

\begin{prop}\label{quiv}
If $A$ is a connected basic algebra with no non-trivial gradings, then its quiver has only one vertex.
\end{prop}
\begin{proof}
Let us assume that the quiver of $A$ has at least two vertices, say $i$ and $j$. Since $A$ is a connected algebra, we can assume that in the quiver of $A$ there is an arrow starting at $i$ and ending at $j$. If $P_i$ and $P_j$ are the corresponding projective indecomposable modules, then the simple module corresponding to $i$ appears as a composition factor in the second radical layer of $P_j$. Hence, there is a non-trivial map from $P_i$ to $P_j$.  If we start from the trivially graded algebra $A$, we get a graded algebra $A^{\prime}={\rm Endgr}_{A}(P_i\langle d_1 \rangle \oplus P_j\langle d_2 \rangle \oplus P\langle d_3 \rangle)^{op}$ that is graded Morita equivalent to $A$. Again, if we choose $d_1$ and $d_2$ to be different integers, then $A^{\prime}$ is a non-trivially graded algebra, because a map from $P_i\langle d_1 \rangle$  to  $P_j\langle d_2 \rangle$ of highest possible rank is a homogeneous element  of degree $d_1-d_2\neq 0$. 
\end{proof}

Thus, in order for a connected finite-dimensional algebra to possess no non-trivial gradings, it has to be basic, and its quiver has to have one vertex. On the other hand, it is not difficult to find a basic algebra whose quiver has only one vertex, that can be non-trivially graded. 
\begin{exa}
The algebra 
defined by the following quiver and relations 
\begin{itemize}\item[] 
\begin{multicols}{2}
$
\hspace{-25mm}\xymatrix{&&&\stackrel{e}{\bullet}\ar@(dl,ul)^{\alpha}\ar@(dr,ur)_{\beta}}\vspace{15mm}
$

$\hspace{-20mm}\alpha^2=0=\beta^2$, $(\alpha\beta)^r=(\beta \alpha)^r,$
\end{multicols}
\end{itemize} 
where  $r$ is a positive integer,  has infinitely many non-trivial gradings indexed by $\mathbb{Z}^2$ (for any pair $(d_1,d_2)$ of integers define $\deg(e)=0$, $\deg(\alpha)=d_1$, and $\deg(\beta)=d_2$ to get a grading on this algebra), with no two of them graded Morita equivalent (cf.\ \cite[Proposition 5.1]{BDeh}).   
\end{exa}

Recall that a grading on a finite
dimensional algebra $A$ can be seen as a cocharacter $\pi \, : \,
\textbf{G}_m\rightarrow \, {\rm Aut}(A)$. We will use the same
letter $\pi$ to denote the corresponding cocharacter of ${\rm
Out}(A)$, the group of outer automorphisms of $A$ (i.e.\ the quotient of ${\rm Aut}(A)$ by the subgroup of all inner automorphisms), which is given by the composition of $\pi$ and the canonical
surjection.
The image of the cocharacter $\pi$  is contained in a maximal torus of ${\rm Aut}(A)$.  Tori are contained in ${\rm Aut}(A)^0$, the
connected component of ${\rm Aut}(A)$ that contains the identity element (see \cite{BRL} for more details on algebraic groups). Cocharacters $\pi$ and $\pi^{\prime}$ of ${\rm Aut}(A)$ are conjugate if
there exists $g \in {\rm Aut}(A)$ such that $\pi^{\prime}(x)=g\pi g^{-1}(x)$
for all $x\in \textbf{G}_m$.

The following proposition tells us how to classify all gradings on
$A$ up to graded Morita equivalence. 

\begin{prop}[{\cite[Corollary 5.9]{Rou}}]\label{cokarak}
Two basic graded algebras $(A,\pi)$ and $(A,\pi^{\prime})$ are graded
Morita equivalent if and only if the corresponding cocharacters
$\,\pi \, : \, \textbf{G}_m \rightarrow {\rm Out}(A)$  and
$\,\pi^{\prime} \, : \, \textbf{G}_m \rightarrow {\rm Out}(A)$ are
conjugate.
\end{prop}
From this proposition we see that in order to classify gradings on
$A$ up to graded Morita equivalence, we need to compute maximal
tori in  ${\rm Out}(A)$.

\begin{thm} \label{Glavna}
Let $A$ be a connected finite dimensional algebra over $k$. The algebra $A$ can not be non-trivially graded if and only if $A$ is basic,  its quiver has one vertex, and its group of outer automorphisms ${\rm Out}(A)$ is unipotent. 
\end{thm}
\begin{proof} If there are no non-trivial gradings on $A$, then by Proposition \ref{basic} and Proposition \ref{quiv} we have that $A$ has to be basic and its quiver has to have only one vertex. If the group ${\rm Out}(A)$ is not unipotent, then its maximal tori would be non-trivial and there would exist a non-trivial cocharacter $\,\pi \, : \, \textbf{G}_m \rightarrow {\rm Out}(A)$. By Lemma 5.5 in 
\cite{Rou}, this cocharacter can be lifted to a non-trivial cocharacter $\, \tilde{\pi} \, : \, \textbf{G}_m \rightarrow {\rm Aut}(A)$. This non-trivial cocharacter gives us a non-trivial grading on $A$.  Thus, ${\rm Out}(A)$ has to be unipotent. 

Assume now that  $A$ is basic,  its quiver has one vertex, and its group of outer automorphisms ${\rm Out}(A)$ is unipotent. Since  ${\rm Out}(A)$ is unipotent, any two gradings on $A$ correspond to the same, trivial cocharacter.  This means that their corresponding cocharacters are conjugate in ${\rm Out}(A)$, and that any two gradings on $A$ must be graded Morita equivalent. In particular, any grading on $A$ is graded Morita equivalent to the trivial grading. Because $A$ is basic with only one simple module, a graded algebra that is graded Morita equivalent to the trivially graded algebra $A$ is ${\rm Endgr}_{A}(A\langle d \rangle)^{op}$, with $d$ being an integer. But this graded algebra is again the trivially graded algebra. 
\end{proof}

\begin{rem} \label{rimark}
By Remark 5.10 in \cite{Rou}, Proposition \ref{cokarak} also  holds true if we remove the assumption that $A$ is a basic algebra and replace conjugacy of cocharacters in ${\rm Out}(A)$ by the conjugacy in ${\rm Pic}(A)$, the Pickard group of $A$, which contains ${\rm Out}(A)$ (see \cite{Sao} and  \cite{Rou} for details on Pickard group). 
If $A$ is a connected finite dimensional algebra such that ${\rm Out}(A)$ is a unipotent group, then for any two gradings on $A$, the corresponding cocharacters of ${\rm Out}(A)$ are obviously conjugate  both in ${\rm Out}(A)$ and ${\rm Pic}(A)$, because they are both equal to the trivial cocharacter. Hence, any grading on $A$ is graded Morita equivalent to the trivial grading.  
\end{rem}

\begin{rem} 
Let ${\rm Out}^K (A)$ be  the subgroup of
${\rm Out}\, (A)$ of those automorphisms fixing the isomorphism
classes of simple $A$-modules. This is a subgroup consisting of those automorphisms $f$ such that $S^f\cong S$, for every simple $A$-module $S$, where $S^f$ is an $A$-module whose underlying set is $S$ and $a\in A$ acts on an element $s\in S$ by $a*s=f(a)s$.   Since ${\rm Out}^K (A)$ contains ${\rm Out}^0(A)$, 
the connected component of ${\rm Out}(A)$ that
contains the identity element, we have that maximal tori in ${\rm
Out}(A)$ are contained in ${\rm Out}^K (A)$, i.e.\  if ${\rm Out}^K(A)$ is a unipotent group, then ${\rm Out}(A)^0$ is a unipotent group as well.  It is often the case, in particular for a basic algebra given by a quiver and relations, that it is easier to compute the group ${\rm Out}^K(A)$ then ${\rm Out}(A)$.  
\end{rem}

\section{Quaternion blocks of group algebras}

A block of a group algebra over a field of characteristic $p$ is of tame representation type if and only if $p=2$ and its defect group is a
dihedral, semidihedral, or generalized quaternion group.  If $B$ is a block of a finite group $G$ over an algebraically closed field $k$ of characteristic $2$ with one isomorphism
class of simple modules and with a quaternion defect group of
order $2^{n-2}$ for some positive integer $n$, then $B$ is Morita
equivalent to the algebra $Q(1E)^r$ from Erdmann's classification of tame blocks \cite{Er2}, where $r=2^{n-2}$, and  
$Q(1E)^r$ is the algebra
defined by the quiver and relations
\begin{multicols}{2}
$
\hspace{-10mm}\xymatrix{&&&\stackrel{e}{\bullet}\ar@(dl,ul)^{\alpha}\ar@(dr,ur)_{\beta}}\vspace{15mm}
$
$\hspace{5mm}\begin{aligned}
\alpha^2&=(\beta\alpha)^{r-1}\beta, (\alpha\beta)^r=(\beta \alpha)^r,\\ 
  \beta^2&=(\alpha\beta)^{r-1}\alpha, (\alpha\beta)^r\alpha=0.\end{aligned}$
\end{multicols}
Thus, any block with a quaternion defect group and one isomorphism class
of simple modules is Morita equivalent to some algebra $Q(1E)^r$ from the
above family of algebras.

We will now prove that $Q(1E)^r$ does not admit a non-trivial
grading.  If we try to take the most obvious generating set consisting of arrows and the vertex of the quiver of $Q(1E)^r$ as a homogenous set of generators, then it would have to hold that $2\deg(\alpha)=r\deg(\beta)+(r-1)\deg(\alpha)$ and $2\deg(\beta)=r\deg(\alpha)+(r-1)\deg(\beta)$, with $\deg(\alpha)$ and $\deg(\beta)$ denoting the degrees of $\alpha$ and $\beta$. It is obvious that this system of linear equations does not have non-trivial integer solutions. If we try some other generating set, we will end up with the same result, we will not  be able to construct a non-trivial grading. We will prove that there is no non-trivial grading by showing that every outer automorphism of
$Q(1E)^r$ is unipotent. It will follow that the maximal torus of
${\rm Out}(Q(1E)^r)$ is the trivial group.

If $\varphi$ is an arbitrary automorphism in ${\rm Out}(Q(1E)^r)$, then  $\{e\}$, the set of the vertices of the quiver of $Q(1E)^r$, is a complete set of
primitive orthogonal idempotents. Also, the set
$\{\varphi(e)\}$ is a complete set of
primitive orthogonal idempotents. From classical ring theory (e.g.\
\cite[Theorem 3.10.2]{Jac}) we know that there exists an
invertible element $x$ such that
$x^{-1}\varphi(e)x=e$. Since we work inside ${\rm Out}(Q(1E)^r)$, we can assume that $\varphi(e)=e$, and that $\varphi$ fixes the isomorphism classes of simple modules.

Since $\varphi({\rm rad}\, Q(1E)^r)\subseteq {\rm rad}\, Q(1E)^r $, for a given arrow $t$ in the
quiver of $Q(1E)^r$, $\varphi(t)$ is a linear combination of paths of length greater than 0.

Thus, we can assume that
$$
\begin{array}{c@{=}l}
\varphi(e)&e,\\
\varphi(\alpha)&a_1\alpha+a_2\beta +a_3x,\\
\varphi(\beta)&b_1\alpha+b_2\beta +b_3y,
\end{array}
$$
where $a_i,b_i\in k$, and $x,y\in {\rm rad}^2Q(1E)^r$.  Direct computation shows that
$(\varphi(\alpha)\varphi(\beta))^r=(a_1b_2)^r(\alpha\beta)^r+(a_2b_1)^r(\beta\alpha)^r$.
From the relations of $Q(1E)^r$ we have that
$(\varphi(\alpha)\varphi(\beta))^r=\varphi(\alpha^3)=\varphi(\beta^3)$.
It follows that $a_1^2a_2=0$ and $b_1b_2^2=0$. Because
$(\varphi(\alpha)\varphi(\beta))^r\neq 0$, then either $a_1\neq
0\neq b_2$ and $a_2=b_1=0$, or $a_2\neq 0\neq b_1$ and
$a_1=b_2=0.$

If the former holds, then from
$\varphi(\alpha^2)=\varphi(\beta)(\varphi(\alpha)\varphi(\beta))^{r-1}$
it follows after an elementary but tedious computation that
$a_1^2=a_1^{r-1}b_2^r.$ Similarly, from
$\varphi(\beta^2)=\varphi(\alpha)(\varphi(\beta)\varphi(\alpha))^{r-1}$
it follows that $b_2^2=a_1^rb_2^{r-1}$. Subsequently, we have that
$a_1^3=b_2^3$. From $a_1^3=a_1^{r}b_2^r$ and $a_1^3=b_2^3$, it now follows that there exists an even integer
$m$ such that $a_1^m=b_2^m=1$, e.g.\ $m=12r-18$ (we get this number by taking the cube of both sides of $a_1^3=a_1^{r}b_2^r$ and using that  $a_1^3=b_2^3$) . For such an integer $m$ we have
that
$$(\varphi-{\rm id}_{Q(1E)^r})^m=\varphi^m-{\rm id}_{Q(1E)^r}.$$
Since $\varphi^m-{\rm id}_{Q(1E)^r}$ is a nilpotent homomorphism,
it follows that $\varphi$ is a unipotent homomorphism.

If $a_2\neq 0\neq b_1$ and $a_1=b_2=0$, then as in the previous
case, using the same arguments we conclude that
$a_2^rb_1^{r-3}=1$, $a_2^{r-3}b_1^{r}=1$ and $a_2^3=b_1^3.$ Now,
$\varphi^2$ is a unipotent homomorphism, by the previous
paragraph. Since we work over a field of characteristic 2, it
follows that $(\varphi- {\rm id}_{Q(1E)^r})^2$ is a nilpotent
homomorphism. Hence, $\varphi$ is a unipotent homomorphism.

\begin{thm}
The group ${\rm Out}(Q(1E)^r)$ is unipotent. There are no non-trivial gradings on $Q(1E)^r$.
\end{thm}
\begin{proof} 
From the above computation it follows that ${\rm Out}(Q(1E)^r)$ is unipotent. By Theorem \ref{Glavna}, there are no non-trivial gradings on $Q(1E)^r$. 
\end{proof}

\begin{prop}
Let $B$ be a tame block of group algebras which is Morita equivalent to $Q(1E)^r$.
The group ${\rm Out}(B)$ is a unipotent group. 
\end{prop}
\begin{proof}
By \cite[Theorem 17]{Sao}, ${\rm Out}^0(B)$ is invariant under derived and stable 
equivalence, hence,  it is invariant under Morita equivalence. Since the maximal tori in ${\rm Out}^0(Q(1E)^r)$ are trivial, it must be that the maximal tori in ${\rm Out}^0(B) $ are also trivial,  which implies that ${\rm Out}(B)$ is a unipotent group.
\end{proof}

\begin{cor} Let $B$ be a tame block of group algebras with quaternion defect groups and one isomorphism class of simple modules.  Up to graded Morita equivalence, there does not exist a non-trivial grading on  $B$. 
\end{cor}
\begin{proof} From the previous proposition, Theorem \ref{Glavna} and Remark \ref{rimark} it follows that every grading on $B$ is graded Morita equivalent to the trivial grading on $B$. 
\end{proof}

We note here that we can not claim that $B$ does not have non-trivial gradings, even though $Q(1E)^r$ does not have non-trivial gradings. The key difference is that $B$ does not have to be a basic algebra.

\end{document}